\documentclass[12pt]{amsart}
\usepackage{latexsym,amsmath, amscd, amssymb, amsthm, euscript, mathrsfs, stmaryrd}
\usepackage[all]{xy}
\usepackage{hyperref}
\usepackage[colorinlistoftodos]{todonotes}
\usepackage[cal=boondoxo, scr=boondoxo]{mathalfa}
\usepackage{graphicx}
\graphicspath{ {./images/} }

\newtheorem{thm}[subsection]{Theorem}

\newtheorem{thm/def}[subsection]{Theorem/Definition}

\newtheorem{lem}[subsection]{Lemma}

\newtheorem{prop}[subsection]{Proposition}

\theoremstyle{definition}

\newtheorem{defn}[subsection]{Definition}

\theoremstyle{definition}

\theoremstyle{definition}

\newtheorem{rem}[subsection]{Remark}

\newtheorem{example}[subsection]{Example}
\numberwithin{equation}{subsection}
\newtheorem{claim}[subsection]{Claim}
\newtheorem{data}[subsection]{Data}
\newtheorem{pg}[subsection]{}

\newcommand{\Sp}{\text{\rm Spec}}

\newcommand{\mls}{\mathscr}

\usepackage{tikz-cd}

\usepackage{rotating}

\AtBeginDocument{%
   \def\MR#1{}
}

\begin{document}

\title{Point objects on abelian varieties}

 \author{Aise Johan de Jong and Martin Olsson}

\begin{abstract}
We classify the point objects in the derived category $D(X)$ of a torsor under an abelian variety over a field of characteristic $0$.
\end{abstract}

\maketitle

\section{Introduction}
Let $k$ be a field of characteristic $0$ and let $X/k$ be a torsor under an abelian variety $A/k$ of dimension $d$.  Let $A^t$ denote the dual abelian variety of $A$ and let $D(X)$ denote the bounded derived category of coherent sheaves on $X$. 

Recall (see for example \cite{MR498572}) that if $k$ is algebraically closed then a vector bundle $\mls E$ on $X$ is called \emph{semi-homogeneous} if for every $a\in A(k)$ there exists a line bundle $\mls L_a$ on $X$ such that
$$
t_a^*\mls E\simeq \mls E\otimes \mls L_a.
$$
For general fields $k$ we call a vector bundle $\mls E$ semi-homogeneous if its base change $\mls E_{\bar k}$ to $X_{\bar k}$ is semi-homogeneous, where $\bar k$ is an algebraic closure of $k$.

The main result of this article is the following:

\begin{thm}\label{T:mainthm}
Let $F\in D(X)$ be an object satisfying the following:
\begin{enumerate}
    \item [(i)] $\text{\rm Ext}^i(F, F) = 0$ for $i<0$.
    \item [(ii)] The $k$-vector space $\text{\rm Ext}^0(F, F)$ has dimension $1$ and $\text{\rm Ext}^1(F, F)$ has dimension $\leq d$.
\end{enumerate}
Then 
$$
F \simeq i_*\mls E[r]
$$
where $r$ is an integer,  $i:Z\hookrightarrow X$ is a torsor under a sub-abelian variety $H\subset  A$, and $\mls E$ is a geometrically simple semi-homogeneous (with respect to the $H$-action) vector bundle on $Z$.

Conversely, if $i:Z\hookrightarrow X$ is a torsor under a sub-abelian variety $H\subset A$ and $\mls E$ is a geometrically simple semi-homogeneous vector bundle on $Z$ then $F:= i_*\mls E$ satisfies conditions (i) and (ii).
\end{thm}


\begin{rem} If $k$ is algebraically closed then there is a classification of simple semi-homogeneous vector bundles on an abelian variety $A$ \cite[5.6 and 5.8]{MR498572}.  They are all of the form $\pi _*\mls L'$, where $\pi :A'\rightarrow A$ is an isogeny for which the induced map
$$
\lambda :A'\rightarrow A^{\prime t}, \ \ a'\mapsto t_{a'}^*\mls L'\otimes \mls L^{\prime -1}
$$
is injective on $\text{Ker}(\pi )$.  This description will be important in our arguments and we revisit some of the arguments proving this in a more general setting in sections \ref{S:section2}-\ref{S:section4}.
\end{rem}

\begin{rem} The assumption that $k$ has characteristic $0$ is used in exactly one place in the proof (see \ref{P:6.1}).  The rest of the argument does not use this assumption.
\end{rem}

\begin{rem} After completion of the initial draft of this article we were made aware of the close connection between this work and that of Polishchuk in \cite{Polishchuk, Polishchuk2, Polishchuk3}.  In fact, the methods proving \ref{T:mainthm} also yield a classification of ``Lagrangian pairs'' studied in loc. cit.  We explain this in section \ref{S:section7}.
\end{rem}

\begin{example}
Let $A$ be an abelian variety over $k$ and let $\mls L$ be a line bundle on $A$.  Let $\Phi :D(A)\rightarrow D(A^t)$ be the equivalence provided by the Poincar\'e bundle.  Then $\Phi (\mls L)$ is a point object on $A^t$ which is described in  \cite[5.1]{Gulbrandsen} as follows. Let
$$
\lambda :A\rightarrow A^t, \ \ a\mapsto t_a^*\mls L\otimes \mls L^{-1}
$$
be the map defined by $\mls L$, let $C\subset A$ be the connected component of the identity in $\text{Ker}(\lambda )$, and let $q:A\rightarrow B:= A/C$ be the induced quotient.  By a theorem of Kempf \cite[Appendix, Theorem 1]{Quadratic} we can write the line bundle $\mls L$ as
$$
\mls L\simeq q^*\mls M\otimes \mls S
$$
for a non-degenerate line bundle $\mls M$ on $B$ and a line bundle $\mls S$ algebraically equivalent to $0$.  Gulbrandsen then shows that 
$$
\Phi (\mls L) = t_{\mls S}^*(q^t_*\Phi _B(\mls M)),
$$
where $\Phi _B:D(B)\rightarrow D(B^t)$ is the Fourier transform for $B$ and $q^t:B^t\rightarrow A^t$ is the dual of $B$.  Now since $\mls M$ is non-degenerate the transform $\Phi _B(\mls M)$ is a shift of a vector bundle on $B^t$ \cite[11.11]{Polishchuk} and we see that $\Phi (\mls L)$ has the form described in \ref{T:mainthm}.
\end{example}

\subsection{Idea of the proof}

The basic idea of the proof of \ref{T:mainthm} is to study the action of $A\times A^t$ on the derived category $D(X)$ provided by Rouquier's work \cite{Rouquier}.  If $F$ satisfies conditions (i) and (ii) in \ref{T:mainthm} then we show that the stabilizer $\mathbf{S}\subset A\times A^t$, suitably defined, of $F$ under this action has dimension equal to $\text{dim}(A)$.  This gives $F$ the structure of a ``semi-homogeneous complex'' under the action of $\mathbf{S}$ in the sense that for $(a, [\mls L])\in \mathbf{S}$ we have
$$
t_a^*F\simeq F\otimes \mls L.
$$
The rest of the proof then reduces to understanding the structure of such complexes.

Sections \ref{S:section2}-\ref{S:section4} are devoted to developing the necessary theory of semi-homogeneous complexes and proving some technical results that will be needed in the proof.  In the case of vector bundles, a reference for the material in this section is \cite{MR498572}. Then in section \ref{S:section5} we collect some basic results about moduli of point objects in derived categories which enables us to make precise the meaning of the stabilizer $\mathbf{S}$.  Finally in section \ref{S:section6} we put the technical ingredients together and complete the proof of \ref{T:mainthm}.  As noted above we also include a discussion in section \ref{S:section7} of the relationship with Polishchuk's work on Lagrangian pairs.

\subsection{Acknowledgements}
The authors thank A. Polishchuk for helpful comments including pointing out  the connection with the work in \cite{Polishchuk, Polishchuk2, Polishchuk3}.
de Jong partially supported by 
 NSF grant DMS-2052750 FRG: Collaborative Research: Derived Categories, Moduli
Spaces, and Classical Algebraic Geometry, and Olsson partially supported by NSF
grant DMS-1902251 Derived categories and other invariants of algebraic varieties.

\section{Homogeneous complexes}\label{S:section2}

Let $A$ be an abelian variety over an algebraically closed  field $k$.

\begin{defn} A \emph{homogeneous complex} on $A$ is a nonzero object $K\in D(A)$ such that for every $a\in A(k)$ we have
$$
t_a^*K\simeq K,
$$
where $t_a:A\rightarrow A$ is translation by $a$.
\end{defn}

In the case when $K$ is a sheaf we obtain the usual notion of a homogeneous vector bundle \cite{MR498572}.

\begin{pg}
Note that if $K$ is a homogeneous complex then each cohomology sheaf $\mls H^i(K)$ is a homogeneous vector bundle on $A$.  Let
\begin{equation}\label{E:1.2.1}
\Phi :D(A)\rightarrow D(A^t)
\end{equation}
be the equivalence of derived categories given by the Poincar\'e bundle.  Then $\Phi (\mls H^i(K))$ is set-theoretically supported on a finite set of points of $A^t$ \cite[4.19]{MR498572}.  It follows that we have a decomposition
$$
\Phi (K) = \oplus _{{\mls L}\in A^t(k)} K^t_{\mls L},
$$
where $K^t_{\mls L}$ is a complex on $A^t$ supported at $[{\mls L}]\in A^t(k)$.  Applying the inverse transform we have
\begin{equation}\label{E:1.2.2}
K \simeq \oplus _{{\mls L}\in A^t(k)}K_{\mls L},
\end{equation}
where 
$$
K_{\mls L}:= \Phi ^{-1}(K^t_{\mls L}).
$$
Now observe that $K_{\mls L}^t$ admits a finite filtration whose successive quotients are of the form $\kappa ([{\mls L}])[s]$ for various integers $s$, and therefore we can endow the complex $K_{\mls L}$ with the structure of a filtered complex whose graded pieces are of the form ${\mls L}[s]$ (note that this structure is not canonical).    Setting
$$
U_{\mls L}:= K_{\mls L}\otimes {\mls L}^{-1}
$$
we obtain a decomposition
\begin{equation}\label{E:decomposition}
K \simeq \oplus _{{\mls L}\in A^t(k)}{\mls L}\otimes U_{\mls L},
\end{equation}
where $U_{\mls L}$ admits a filtration whose successive quotients are of the form $\mls O_A[s]$.
\end{pg}

\begin{lem}\label{L:1.3} For ${\mls L}\neq {\mls L}'$ we have
\begin{equation}\label{E:1.3.1}
\text{\rm Ext}^*({\mls L}\otimes U_{\mls L}, {\mls L}'\otimes U_{{\mls L}'}) = 0,
\end{equation}
and therefore
$$
\text{\rm Ext}^*(K, K)\simeq \oplus _{{\mls L}\in A^t(k)}\text{\rm Ext}^*({\mls L}\otimes U_{\mls L}, {\mls L}\otimes U_{\mls L})\simeq \oplus _{{\mls L}\in A^t(k)}\text{\rm Ext}^*(U_{\mls L}, U_{\mls L}).
$$
\end{lem}
\begin{proof}
To prove the vanishing \eqref{E:1.3.1} note that since $U_{\mls L}$ and $U_{{\mls L}'}$ admit filtrations whose successive quotients are of the form $\mls O_A[s]$ it suffices to show that 
$$
\text{\rm Ext}^*({\mls L}, {\mls L}')\simeq H^*(A, {\mls L}^{-1}\otimes {\mls L}') = 0,
$$
which follows from the fact that ${\mls L}^{-1}\otimes {\mls L}'$ is a nontrivial line bundle algebraically  equivalent to $0$.
\end{proof}

\begin{lem} Let $K$ be a homogeneous complex on $A$.

(i) Each $U_{\mls L}$ in the decomposition above is, up to a shift, a sheaf if and only if $\text{\rm Ext}^i(K, K) = 0$ for $i<0$.

(ii) If $\text{\rm Ext}^i(K, K) = 0$ for $i<0$ and $\text{\rm Ext}^0(K, K) = k$ then $K$ is a line bundle, up to a shift.
\end{lem}
\begin{proof}
The ``only if'' part of statement (i) is immediate.    To see the ``if'' direction we proceed as follows.
With notation as in \eqref{E:1.2.2} we have
$$
\text{\rm Ext}^*(U_{\mls L}, U_{\mls L})\simeq \text{\rm Ext}^*_{A^t}(K_{\mls L}^t, K^t_{\mls L}).
$$
Assume that $K_{\mls L}^t$ is concentrated in more than one degree, and let $a$ (resp. $b$) be the degree of the top (resp. bottom) cohomology of $K_{\mls L}^t$.  So $a>b$ and $H^a(K_{\mls L}^t)$ and $H^b(K_{\mls L}^t)$ are finite length modules supported at the point $[{\mls L}]\in A^t$.  Since these are finite length modules over the same point there exists a nonzero morphism
$$
\varphi :H^a(K_{\mls L}^t)\rightarrow H^b(K_{\mls L}^t).
$$
Such a morphism defines a nonzero morphism
$$
K^t_{\mls L}\rightarrow K^t_{\mls L}[b-a],
$$
or equivalently a nonzero element of
$$
\text{\rm Ext}^{b-a}(K^t_{\mls L}, K^t_{\mls L})
$$
contradicting the vanishing of the negative $\text{\rm Ext}$-groups.  It follows that $K_{\mls L}^t$ is supported in a single degree. By \cite[4.12]{MR498572} it follows that $K_{\mls L}$, and therefore also $U_{\mls L}$, is also supported in a single degree proving (i).

To see (ii) note that if $K_{\mls L}^t$ has length $>1$ then there exists non-scalar endomorphisms of $K_{\mls L}^t$ and therefore also non-scalar endomorphisms of $K_{\mls L}$.  Therefore each $K^t_{\mls L}$ is, up to a shift, the skyscraper sheaf associated to the point $[{\mls L}]$, and $K_{\mls L} \simeq {\mls L}$, up to a shift. Furthermore, since $\text{Ext}^0(K, K) = k$ there is only one factor in the decomposition \eqref{E:decomposition} proving (ii).
\end{proof}

\section{Semi-homogeneous complexes}\label{S:section3}

\begin{pg}
Let $k$ be an algebraically closed field and fix a diagram of abelian varieties over $k$
\begin{equation}\label{E:setup}
\xymatrix{
\mathbf{S}\ar[d]_-{u}\ar[r]^-{v}& A^t\\
A,&}
\end{equation}
where $u$ is surjective.
\end{pg}

\begin{defn}\label{D:2.1} A complex $K\in D(A)$ is \emph{semi-homogeneous with respect to \eqref{E:setup}} if for every  $s\in \mathbf{S}(k)$ we have
\begin{equation}\label{E:2.2.1}
t_{u(s)}^*K\simeq K\otimes {\mls L}_{v(s)},
\end{equation}
where $\mls L_{v(s)}$ is the line bundle corresponding to $v(s)$.
\end{defn}

To study semi-homogeneous complexes we will consider additional geometric structure on \eqref{E:setup} as follows (a priori such additional data may not exist but see \ref{L:dataexist} below).  
\begin{data}\label{P:3.2} 
(i) A commutative diagram
\begin{equation}\label{E:2.5.1}
\xymatrix{
\mathbf{S}\ar[rd]_-{\pi ^t\circ v}\ar[r]^-{u'}\ar@/^2pc/[rr]^-u& A'\ar[d]^-{\alpha '}\ar[r]^-\pi & A\\
& A^{\prime t},&}
\end{equation}
where $\pi $ is an isogeny and $\alpha ':A'\rightarrow A^{\prime t}$ is a morphism of abelian varieties.

(ii) A line bundle $\mls N'$ on $A'$ such that the map $\alpha '$ is equal to the map
$$
\lambda _{\mls N'}:A'\rightarrow A^{\prime t}, \ \ a'\mapsto t_{a'}^*\mls N'\otimes \mls N^{\prime -1}.
$$
\end{data}

\begin{pg}\label{P:3.4}
Let $K$ be a nonzero semi-homogeneous complex with respect to \eqref{E:setup}.  Then for every integer $i$ the sheaf $\mls H^i(K)$ is a semi-homogeneous sheaf and therefore isomorphic to a vector bundle, since the semi-homogeneous condition implies that the maximal open in $A$ over which $\mls H^i(K)$ is locally free is translation invariant.  Fix an integer $i$ for which $\mls H^i(K)$ is nonzero, let $r$ be its rank, and let $\mls M$ denote the determinant of $\mls H^i(K)$.  Taking determinants of the isomorphism \eqref{E:2.2.1} we then obtain that for every $s\in \mathbf{S}(k)$ we have
$$
t_{u(s)}^*\mls M\simeq \mls M\otimes {\mls L}_{v(s)}^{\otimes r}.
$$
In other words, if 
$$
\lambda _{\mls M}:A\rightarrow A^t, \ \ a\mapsto t_a^*\mls M\otimes \mls M^{-1}
$$
denotes the homomorphism defined by $\mls M$ then the diagram
\begin{equation}\label{E:2.4.1}
\xymatrix{
\mathbf{S}\ar[d]^-u\ar[r]^-{v}& A^t\ar[d]^-{[r]}\\
A\ar[r]^-{\lambda _{\mls M}} & A^t}
\end{equation}
commutes, where $[r]:A^t\rightarrow A^t$ denotes multiplication by $r$.
\end{pg}

\begin{lem}\label{L:dataexist} If there exists a non-zero complex $K$ semi-homogeneous with respect to \eqref{E:setup} then there exists data as in \ref{P:3.2}.
\end{lem}
\begin{proof} We proceed with notation as in \ref{P:3.4}.
To obtain the diagram \eqref{E:2.5.1} in (i), let $A'$ be the connected component of the identity in
$$
A\times _{{\lambda _{\mls M}} , A^t, [r]}A^t,
$$
let $\alpha :A'\rightarrow A^t$ (resp. $\pi $) be the map given by the second (resp. first) projection, and let $u'$ be the map induced by the commutative square \eqref{E:2.4.1}.  Setting $\alpha':= \pi ^t\circ \alpha $ we then get the desired commutative diagram.

For the existence of $\mls N'$ in \ref{P:3.2} (ii) note that the composition
$$
\xymatrix{
A'\ar[r]^-{\alpha '}&A^{\prime t}\ar[r]^-{[r]}& A^{\prime t}}
$$
is equal to the map $\lambda _{\mls M'}$, where $\mls M':= \pi ^*\mls M$.  Indeed the precomposition of this map with the surjective $u':\mathbf{S}\rightarrow A'$ is the map (using the commutativity of \eqref{E:2.4.1} and \eqref{E:2.5.1})
$$
[r]\circ \pi ^t\circ v = \pi ^t\circ [r]\circ v = {\lambda _{\mls M}} \circ u.
$$
It follows that $A'[r]\subset \text{Ker}(\lambda _{\mls M'})$, and therefore by \cite[Theorem 3 on p. 214]{Mumford}  the line bundle $\mls M'$ has an $r$-th root $\mls N'$.  The two maps 
$\lambda _{\mls N'}, \alpha ':A'\rightarrow A^{\prime t}$
have equal compositions with $[r]:A^{\prime t}\rightarrow A^{\prime t}$, and therefore must be equal.
\end{proof}

\begin{lem}\label{L:2.5}
Fix \eqref{E:setup} and data \ref{P:3.2}, and let $K$ be a semi-homogeneous complex with respect to \eqref{E:setup}.

(i)  For any $a'\in A'(k)$ we have 
$$
t_{a'}^*\pi ^*K\simeq \pi ^*K\otimes {\mls L}_{\alpha '(a')}.
$$

(ii) There is a decomposition
\begin{equation}\label{E:2.5.2}
\pi ^*K\simeq \oplus _{{\mls L}\in A^{\prime t}(k)}U_{\mls L}\otimes {\mls L}\otimes \mls N',
\end{equation}
where $U_{\mls L}$ is a homogeneous complex admitting a filtration whose successive quotients are of the form $\mls O_{A'}[s]$ for various integers $s$.
\end{lem}
\begin{proof}  
To see (i), note first of all that $u'$ is surjective since the composition $\pi \circ u' = u$ is surjective by assumption. Given $a'\in A'(k)$ with $\pi (a') = a$ choose $s\in \mathbf{S}(k)$ with $u'(s) = a'$.  Since $K$ is semi-homogeneous with respect to \eqref{E:2.5.1} we then have 
$$
t_{a}^*K\simeq K\otimes \mls L_{v(s)}.
$$
Applying $\pi ^*$ and noting that $\pi ^*t_a^*K\simeq t_{a'}^*\pi ^*K$ and $\pi ^*\mls L_{v(s)}\simeq \mls L_{\pi ^t\circ v(s)}$ we find that
$$
t_{a'}^*\pi ^*K\simeq \pi ^*K\otimes \mls L_{\pi ^t\circ v(s)}.
$$
Statement (i) then follows from the commmutativity of the left triangle in \eqref{E:2.5.1} which implies that
$$
\mls L_{\pi ^t\circ v(s)}\simeq \mls L_{\alpha '(a')}.
$$

Finally (ii) follows from (i) by considering 
 the decomposition \eqref{E:decomposition} for the  homogeneous complex 
$$
\pi ^*K\otimes \mls N^{\prime -1}
$$
on $A'$.
\end{proof}

For data \ref{P:3.2} we can consider the degree of the isogeny $\pi $, an integer $\geq 1$. 

\begin{lem}\label{L:3.5} For data \ref{P:3.2} with $\pi $ of minimal degree the group scheme
\begin{equation}\label{E:2.6.1}
I:= \text{\rm Ker}(\pi )\cap \text{\rm Ker}(\alpha ')
\end{equation}
is $0$.
\end{lem}
\begin{proof}
Assume to the contrary that the group scheme $I$ is nontrivial.
We show that there exists data \ref{P:3.2} with $\pi $ of smaller degree.  Since $\alpha '$ is defined by a line bundle there is a skew-symmetric pairing (see for example \cite[p. 205]{Mumford})
$$
e:\text{Ker}(\alpha ')\times \text{Ker}(\alpha ')\rightarrow \mathbf{G}_m.
$$

There is a nonzero subgroup scheme $I'\subset I$ such that the restriction of $e$ to $I'$ is identically zero.  This is clear if $I(k)\neq 0$ from the skew-symmetry of $e$: simply choose a nonzero element $x\in I(k)$ and let $I'$ denote the cyclic subgroup generated by $x$.  If $I$ is a local group scheme note that $I$ contains a subgroup scheme $I'\subset I$ of order the characteristic $p$ of $k$.  Indeed $I[p]$ is nonzero since $I$ is local, and is contained in the group scheme $A'[p]$ whose simple constituents are given by $\mathbf{Z}/(p)$, $\mu _p$, and $\alpha _p$.  It then follows from  \cite[Lemma 1, p. 207]{Mumford} that the restriction of $e$ to this $I'$ is zero.  See also \cite[Proof of Lemma 3, p. 216]{Mumford} for a similar argument.

Let $A^{\prime \prime }$ denote $A'/I'$.  By \cite[Theorem 2, p. 213]{Mumford} the line bundle $\mls N'$ descends to a line bundle $\mls N^{\prime \prime }$ on $A^{\prime \prime }$.  Let
$\alpha ^{\prime \prime }=\lambda _{\mls N^{\prime \prime }}:A^{\prime \prime }\rightarrow A^{\prime \prime t}$
denote the induced homomorphism.  Setting $u^{\prime \prime }:\mathbf{S}\rightarrow A^{\prime \prime }$ equal to $u'$ composed with the projection $q:A'\rightarrow A^{\prime \prime }$ we then obtain a commutative diagram
$$
\xymatrix{
\mathbf{S}\ar@/^2pc/[rrr]^-u\ar[r]^-{u'}\ar[rd]_-{\pi ^t\circ v}& A'\ar[r]^-{q}\ar@/_1pc/[rr]_<<<{\pi }\ar[d]^-{\alpha '}& A^{\prime \prime }\ar[d]^-{\alpha ^{\prime \prime }}\ar[r]^-{\pi ^{\prime \prime }}& A\\
& A^{\prime t}& A^{\prime \prime t},\ar[l]_-{q^t}& }
$$
Now the compositions of the two maps
$$
\pi ^{\prime \prime t}\circ v, \alpha ^{\prime \prime }\circ q\circ u':\mathbf{S}\rightarrow A^{\prime \prime t}
$$
with $q^t$ are equal, and therefore these two maps are also equal.  Denoting $q\circ u^{\prime }$ by $u^{\prime \prime }$ we then obtain a commutative diagram
$$
\xymatrix{
\mathbf{S}\ar[rd]_-{\pi ^{\prime \prime t}\circ v}\ar[r]^-{u^{\prime \prime }}\ar@/^2pc/[rr]^-u& A^{\prime \prime }\ar[d]^-{\alpha ^{\prime \prime }}\ar[r]^-{\pi ^{\prime \prime } }& A\\
& A^{\prime \prime t},&}
$$
and $\text{deg}(\pi ^{\prime \prime })< \text{deg}(\pi )$.
\end{proof}

\begin{thm}\label{T:2.7.1} Fix data \ref{P:3.2} with $\text{\rm deg}(\pi )$ minimal, and let $K$ be a nonzero semi-homogeneous complex with respect to \eqref{E:setup}.  Assume further that $\text{\rm Ker}(\pi )$ is an \'etale group scheme.  Then there exists a homogeneous complex $H\in D(A')$ and a line bundle $\mls N'$ on $A'$ such that $K\simeq \pi _*(H\otimes \mls N')$.
\end{thm}

\begin{rem} The assumption on $\text{Ker}(\pi )$ holds for example if the ground field $k$ has characteristic $0$.
\end{rem}

Before starting the proof it is convenient to fix some additional notation. 

\begin{pg}\label{P:3.10}
Let $G$ denote the quotient
$$
A^{\prime t}/\alpha '(\text{Ker}(\pi )).
$$
If $K$ is a semi-homogeneous complex on $A$ with respect to \eqref{E:setup} we then have a decomposition \eqref{E:2.5.2} of $\pi ^*K$.  Let 
$$
\widetilde {\Sigma }(K)\subset A^{\prime t}(k)
$$
be the set of those $\mls L$ for which $U_{\mls L}$ is nonzero, and let
$$
\Sigma (K)\subset G(k)
$$
be the image of $\widetilde {\Sigma }(K)$.
\end{pg}

\begin{lem} The set $\widetilde {\Sigma }(K)$ is the preimage under the projection map $A^{\prime t}(k)\rightarrow G(k)$ of $\Sigma (K)$.
\end{lem}
\begin{proof}
It suffices to show that $\widetilde {\Sigma }(K)$ is stable under translation by $\text{Ker}(\pi )$.  To see this note that 
for $b\in \text{Ker}(\pi )(k)$ we have an isomorphism
$$
t_b^*\pi ^*K = \oplus _{{\mls L}}t_b^*(U_{\mls L}\otimes {\mls L}\otimes \mls N')\rightarrow  \oplus _{{\mls L}}U_{\mls L}\otimes {\mls L}\otimes \mls N'
$$
giving the natural $\text{Ker}(\pi )$-linearization on $\pi ^*K$.  Now 
$$
t_b^*(U_{\mls L}\otimes {\mls L}\otimes \mls N')\simeq (U_{\mls L}\otimes {\mls L}\otimes t_b^*\mls N'\otimes \mls N^{\prime -1})\otimes \mls N'\simeq U_{\mls L}\otimes {\mls L}\otimes {\mls L}_{\alpha '(b)}\otimes \mls N'
$$
so using \ref{L:1.3} we find that $\widetilde {\Sigma }(K)$ is stable under translation by $\alpha '(b)$.
\end{proof}

\begin{proof}[Proof of \ref{T:2.7.1}]
  Since $\alpha '$ is injective on $\text{Ker}(\pi )$ by \ref{L:3.5} we can write $\widetilde {\Sigma }(K)$ as a disjoint union of $\text{Ker}(\pi )$-orbits, say
$$
\widetilde {\Sigma }(K) = \Sigma (K)\times \text{Ker}(\pi ).
$$
Let $H$ denote the homogeneous complex
$$
H:= \oplus _{{\mls L}\in \Sigma (K)}U_{\mls L}\otimes {\mls L}
$$
so
\begin{equation}\label{E:pulliso}
\pi ^*K\simeq \oplus _{b\in \text{Ker}(\pi )}t_b^*(H\otimes \mls N')
\end{equation}
with $\text{Ker}(\pi )$ acting by permuting the factors.  Projecting onto the factor corresponding to $b=0$ we get a map
$$
q:\pi ^*K\rightarrow H\otimes \mls N'
$$
such that the map \eqref{E:pulliso} is described as the sum 
$$
\sum _{b\in \text{Ker}(\pi )}t_b^*(q):\pi ^*K\rightarrow \oplus _{b\in \text{Ker}(\pi )}t_b^*(H\otimes \mls N').
$$
By adjunction the map $q$ gives a map
$$
\phi :K\rightarrow \pi _*(H\otimes \mls N').
$$
This map $\phi $ is an isomorphism.  Indeed to verify this it suffices to show that it is an isomorphism after applying $\pi ^*$ where we recover the isomorphism \eqref{E:pulliso}.
\end{proof}

\section{Further results about semi-homogeneous complexes}\label{S:section4}

\begin{pg}
In  this section we record some additional technical results we will need for our later arguments.  We continue with the notation of the previous section.   That is, we fix a diagram \eqref{E:setup} as well as additional data \ref{P:3.2}. We will assume that $\pi :A'\rightarrow A$ has minimal degree, so that by \ref{L:3.5} the composition
$$
\xymatrix{
\text{Ker}(\pi )\ar@{^{(}->}[r]& A'\ar[r]^-{\alpha '}& A^{\prime t}}
$$
is injective.  As in the previous section let $G$ denote the quotient
$$
A^{\prime t}/\alpha '(\text{Ker}(\pi )).
$$

So for a complex $K$ semi-homogeneous with respect to \eqref{E:setup} we have subsets (see \ref{P:3.10})
$$
\widetilde {\Sigma }(K)\subset A^{\prime t}(k), \ \ \Sigma (K)\subset G(k).
$$
\end{pg}

\begin{rem} If $K$ is a vector bundle semi-homogeneous with respect to \eqref{E:setup} then from \eqref{E:2.5.2} the vector bundle $\pi ^*K$ is graded by $\widetilde {\Sigma }(K)$.  This grading is not preserved by the $\text{Ker}(\pi )$-linearization of $\pi ^*K$, but the coarser grading by $\Sigma (K)$ is preserved.
By descent theory we therefore obtain a decomposition
$$
K = \oplus _{\sigma \in \Sigma (K)}K_\sigma .
$$
We generalize this to complexes in \ref{L:4.6} below.
\end{rem}



\begin{lem}\label{L:4.3} (i) Let $K, K'\in D(A)$ be semi-homogeneous complexes with respect to \eqref{E:setup}  such that
$$
\Sigma (K)\cap \Sigma (K') = \emptyset.
$$
Then 
$$
\text{\rm Ext}^*(K, K') = 0.
$$

(ii) Let $K,K'$ be semi-homogeneous vector bundles with respect to \eqref{E:setup} for which $\Sigma (K) \cap \Sigma (K') \neq \emptyset $.  Then 
$$
\text{\rm Hom}(K, K')\neq 0.
$$
\end{lem}
\begin{proof}
For (i) fix a representative in $A^{\prime t}(k)$ for each element of $\Sigma (K)$ and $\Sigma (K')$.  By \ref{T:2.7.1} we can then write 
$$
K = \pi _*(\oplus _{\sigma \in \Sigma (K)}U_\sigma \otimes {\mls L}_\sigma \otimes \mls N'), \ \ K' = \pi _*(\oplus _{\sigma '\in \Sigma (K')}U_{\sigma '}\otimes \mls L_{\sigma '}\otimes \mls N')
$$
with the $U_\sigma $ and $U_{\sigma '}$ as in \ref{L:2.5} (ii). By adjunction we then have
\begin{eqnarray*}
\text{\rm Ext}^*_A(K, K')& \simeq & \oplus _{b\in \text{Ker}(\pi )}\oplus _{(\sigma , \sigma ')\in \Sigma (K)\times \Sigma (K')}\text{\rm Ext}^*_{A'}(U_\sigma \otimes \mls L_\sigma \otimes \mls L_{\alpha '(b)}\otimes \mls N', U_{\sigma '}\otimes \mls L_{\sigma '}\otimes \mls N')\\
& \simeq & \oplus _{b\in \text{Ker}(\pi )}\oplus _{(\sigma , \sigma ')\in \Sigma (K)\times \Sigma (K')}\text{\rm Ext}^*_{A'}(U_\sigma \otimes \mls L_\sigma \otimes \mls L_{\alpha '(b)}, U_{\sigma '}\otimes \mls L_{\sigma '}).
\end{eqnarray*}
From this and using the filtrations on the $U_\sigma $'s we see that it suffices to show that if $\mls L$ and $\mls L'$ are line bundles on $A'$ corresponding to distinct points of $A^{\prime t}(k)$ then
$$
\text{\rm Ext}^*(\mls L, \mls L') = 0,
$$
which is standard (see for example \cite[7.19]{AV}).

Similarly statement (ii) reduces to showing that if $\mls E$ and $\mls E'$ are unipotent vector bundles on $A'$ then $\text{Hom}(\mls E, \mls E')\neq 0$, which is immediate.
\end{proof}

\begin{pg} Consider a closed immersion $i:A\hookrightarrow B$ of abelian varieties and let $K$ be a semi-homogeneous complex on $A$ with respect to \eqref{E:setup}.  Since $i$ is a homomorphism we have for every $a\in A(k)$ an isomorphism
$$
t_a^*i^*i_*\simeq i^*i_*t_a^*.
$$
From this it follows that $i^*i_*K$ is again a semi-homogeneous complex with respect to \eqref{E:setup}.
\end{pg}

\begin{lem}\label{L:4.5}  Let $K$ be a semi-homogeneous vector bundle on $A$. 
 Then $\Sigma (i^*i_*K) = \Sigma (K).$
\end{lem}
\begin{proof}
Note that by the projection formula we have
$$
i_*i^*i_*K\simeq (i_*\mls O_A)\otimes ^{\mathbf{L}}_{\mls O_B}i_*K\simeq i_*(i^*i_*\mls O_A)\otimes ^{\mathbf{L}}_{i_*\mls O_A}i_*K\simeq i_*((i^*i_*\mls O_A)\otimes K).
$$
Let $\overline B$ denote the quotient $B/A$,  let $g:B\rightarrow \overline B$ be the quotient map, and let $s:\Sp (k)\hookrightarrow \overline B$ be the zero section.  Then
$$
i_*\mls O_A\simeq g^*s_*\mls O_{\Sp (k)},
$$
and we find that
$$
i_*\mls H^s(i^*i_*K)\simeq H^s(k\otimes _{\mls O_{\overline B, s}}k)\otimes _ki_*K.
$$
It follows that we have
$$
\mls H^s(i^*i_*K)\simeq W_s\otimes _kK
$$
for  vector spaces $W_s$, not all of which are zero.
In particular, 
$$
\Sigma (i^*i_*K) = \cup _s\Sigma (\mls H^s(i^*i_*K)) = \Sigma (K).
$$
\end{proof}

\begin{lem}\label{L:4.6} Let $F\in D(B)$ be a complex on $B$ such that for every $s\in \mathbf{S}(k)$ we have
$$
t_{u(s)}^*F\simeq F\otimes \mls R_{s}
$$
in $D(B)$, for a line bundle $\mls R_s$ on $B$ lifting $\mls L_{v(s)}$, and such that for all $s$ we have $\mls H^s(F) \simeq i_*\mls G_s$ for a vector bundle $\mls G_s$ on $A$ semi-homogeneous with respect to \eqref{E:setup}.  Then
$$
F \simeq \oplus _{\sigma \in \Sigma (F)}F_\sigma ,
$$
where 
$$
\Sigma (F):= \cup _s\Sigma (\mls G_s)\subset G(k)
$$
and $F_\sigma $ has the property that $\mls H^s(F_\sigma ) = i_*\mls G_{s, \sigma }.$
\end{lem}
\begin{proof}
Set $F_{\geq s}:= \tau _{\geq s}F$.  We prove by descending induction on $s$ that the lemma holds for $F_{\geq s}$.    The base case is trivial since $F_{\geq s}= 0$ for $s$ sufficiently large.  For the inductive step we assume the result holds for $s$ and prove it for $s-1$.  Considering the distinguished triangle
$$
\xymatrix{
i_*\mls G_{s-1}[-s+1]\ar[r]& F_{\geq s-1}\ar[r]& F_{\geq s}\ar[r]\ar[d]_-{\simeq }& i_*\mls G_{s-1}[-s+2]\ar[d]_-{\simeq }\\
&& \oplus _\sigma F_{\geq s, \sigma }& \oplus _\sigma i_*\mls G_{s-1, \sigma }[-s+2]}
$$
we see that to prove the inductive step it suffices to show that for $\sigma \neq \sigma '$ we have
$$
\text{Ext}^*(F_{\geq s, \sigma }, i_*\mls G_{s-1, \sigma '}) = 0.
$$
Considering the canonical filtration on $F_{\geq s, \sigma }$ we see that for this in turn it suffices to show that for $s\neq t$  
$$
\text{Ext}^*_B(i_*\mls G_{t, \sigma }, i_*\mls G_{s, \sigma '})\simeq \text{Ext}^*_A(i^*i_*\mls G_{t, \sigma }, \mls G_{s, \sigma '}) = 0.
$$
This follows from \ref{L:4.3} (i) and \ref{L:4.5}. 
\end{proof}

\begin{prop}\label{P:4.7} 
Let $F\in D(B)$ be a complex satisfying the following:
\begin{enumerate}
    \item [(i)] For every $a\in A(k)$ with lift $s\in \mathbf{S}(k)$ we have $t_{i(a)}^*F\simeq F\otimes \mls R_s$ for a line bundle $\mls R_s$ on $B$ restricting to the line bundle $\mls L_{v(s)}$ in \ref{D:2.1}.
    \item [(ii)] The complex $F$ is set-theoretically supported on $A$.
    \item [(iii)] $\text{\rm End}(F) = k.$
    \item [(iv)] $\text{\rm Ext}^i(F, F) = 0$ for $i<0$.
\end{enumerate}
Then there exists a line bundle $\mls L$ defining the map $\alpha '$ such that $F\simeq i_*\pi _*\mls L[s]$ for an integer $s$.
\end{prop}
\begin{proof}
Note first of all that conditions (i)-(iii) imply that each cohomology sheaf $\mls H^s(F)$ is of the form $i_*\mls G_s$ for a  vector bundle $\mls G_s$ on $A$ which is semi-homogeneous with respect to \eqref{E:setup}.
To see this it suffices to show that each $\mls H^s(F)$ is scheme-theoretically supported on $A$.  Consider the quotient $B/A$, let 
$$
U = \Sp (R) \subset B/A
$$
be an affine neighborhood of $0$, and let $B_U$ denote the preimage of $U$.  Since $F$ is set-theoretically supported on $A$ the sheaf $\mls H^s(F)$ is the pushforward of its restriction to $B_U$.  If $\mathfrak{m}\subset R$ denotes the maximal ideal corresponding to the origin then it suffices to show that $\mathfrak{m}$ annihilates $\mls H^s(F)|_{B_U}$.  This follows from noting that the action of $R$ on $\mls H^s(F)|_{B_U}$ factors through the map
$$
R\rightarrow \text{Ext}^0_{B_U}(F|_{B_U}, F|_{B_U})\simeq \text{Ext}^0(F, F) = k.
$$

Observe also that in light of \ref{L:4.6} and (iii) we have
$$
\Sigma (F) = \{\sigma \}
$$
for a single element $\sigma \in A'(k)/\alpha '(\text{Ker}(\pi ))$.

Next we show that condition (iv) implies that $F$ is, up to a shift, a sheaf.  Suppose to contrary that $F$ is concentrated in more than one degree and let $n$ (resp. $m$) be the top (resp. bottom) nonzero degree so $n>m$.  Then
$$
\text{Ext}^{m-n}(F, F)\simeq \text{Ext}^0(i_*\mls G_n, i_*\mls G_m),
$$
and the natural map
$$
\text{Ext}^0_A(\mls G_n, \mls G_m)\rightarrow \text{Ext}^0(i_*\mls G_n, i_*\mls G_m)
$$
is injective since a map $\mls G_n\rightarrow \mls G_m$ is determined by its pushforward.  Now since $\Sigma (F)$ is a point we must have $\Sigma (\mls G_n) = \Sigma (\mls G_m)$ so by \ref{L:4.3} (ii) there exists a nonzero map $\mls G_n\rightarrow \mls G_m$, a contradiction.

Thus after possibly applying a shift, $F$ is a sheaf and by \ref{T:2.7.1} we have
$$
F \simeq i_*\pi _*(U\otimes \mls L_\sigma \otimes \mls N')
$$
for a unipotent vector bundle $U$.  This vector bundle $U$ must be of rank $1$, for otherwise it has nonscalar endomorphisms contradicting (iii).  We 
conclude that $F$ has the desired form.
\end{proof}

\section{Moduli of weak point complexes}\label{S:section5}

\begin{pg}
Let $X$ be a smooth projective geometrically connected scheme over a field $k$ of dimension $d$, and let $D(X)$ be its bounded derived category of coherent sheaves.

In the literature one finds various possible notions of point objects in $D(X)$ (see \cite[2.1]{BondalOrlov}, \cite[5.1]{ELO}).  To avoid confusion we introduce the following definition which captures the necessary properties for \ref{T:mainthm}.
\end{pg}

\begin{defn} A \emph{weak point object} is a complex $F\in D(X)$ satisfying the following conditions:
\begin{enumerate}
    \item [(i)] $\text{Ext}^i(F, F) = 0$ for $i<0$.
    \item [(ii)] $\text{Ext}^0(F, F) = k$.
    \item [(iii)] $\text{Ext}^1(F, F)$ has dimension $\leq d$.
\end{enumerate}
\end{defn}

\begin{pg}
For a $k$-scheme $S$ we define a \emph{relative weak point object} on $X_S$ to be an $S$-perfect complex $F\in D(X_S)$ such that for every geometric point $\bar s\rightarrow S$ the induced object
$$
F_{\bar s}\in D(X_{\kappa (\bar s)})
$$
is a weak point object.  

Let $\mls P$ be the fibered category over the category of $k$-schemes which to any $S$ associates the groupoid of relative weak point objects in $D(X_S)$.
\end{pg}

\begin{lem}\label{L:5.4} The fibered category $\mls P$ is an algebraic stack locally of finite type over $k$.
\end{lem}
\begin{proof}
Let $\mls D$ be the fibered category which to any scheme $S$ associates the groupoid of objects $K\in D(X_S)$ such that for all geometric points $\bar s\rightarrow S$ we have $\text{Ext}^i(K_{\bar s}, K_{\bar s}) = 0$ for $i<0$. By \cite[Theorem on p. 176]{Lieblichproper} the fibered category $\mls D$ is an algebraic stack locally of finite type over $k$.  We claim that the natural inclusion
$$
\mls P\hookrightarrow \mls D
$$
is represented by locally closed immersions.   Indeed the substack of $\mls D$ satisfying (iii) is open by semi-continuity and within this substack condition (ii) is represented by locally closed immersions by \cite[\href{https://stacks.math.columbia.edu/tag/0BDL}{Tag 0BDL}]{stacks-project}.
\end{proof}
\begin{rem} Note that $\mls P$ is a $\mathbf{G}_m$-gerbe over an algebraic space, which we will denote by $P$.
\end{rem}

\begin{pg}
 For any object $F\in D(X)$ there is an induced map
$$
\kappa _{F}:HH^1(X)\rightarrow \text{Ext}^1(F, F)
$$
defined as follows, where $HH^1(X)$ denotes the first Hochschild cohomology of $X$.  By definition we have
$$
HH^1(X) = \text{Hom}_{X\times X}(\Delta _{X*}\mls O_X, \Delta _{X*}\mls O_X[1]),
$$
and therefore an element $\alpha \in HH^1(X)$ defines a morphism of functors 
$$
\alpha :\text{id}_{D(X)}\rightarrow \text{id}_{D(X)}[1].
$$
Evaluating on $F$ we get an element
$$
\alpha (F):F\rightarrow F[1],
$$
or equivalently an element of $\text{Ext}^1(F, F)$.  The map $\kappa _{F}$ sends $\alpha $ to $\alpha (F)$.  
\end{pg}
\begin{rem} 
Note that if $x\in X(k)$ is a $k$-point with associated skyscraper sheaf $\mls O_x$ then we have
$$
\text{Ext}^1(\mls O_x, \mls O_x)\simeq T_X(x)
$$
and in characteristic $0$ the map $\kappa _{\mls O_x}$ gets identified, via the HKR isomorphism \cite{Swan}, with a map
$$
H^1(X, \mls O_X)\oplus H^0(X, T_X)\rightarrow T_X(x).
$$  
This map is simply the projection to $H^0(X, T_X)$, followed by the natural map $H^0(X, T_X)\rightarrow T_X(x)$.
\end{rem}

\begin{pg} 
For $F\in \mls P(k)$ the map $\kappa _{F}$ has the following geometric interpretation.  Following \cite[5.4]{LieblichOlsson} let $\mls R_X^0$ denote the fibered category over $k$ whose fiber over a scheme $S$ is the groupoid of objects $Q\in D((X\times X)_S)$ satisfying the "Rouquier condition".  So $\mls R_X^0$ is a $\mathbf{G}_m$-gerbe over 
$$
\mathbf{R}_X^0:= \text{Aut}^0_X\times \text{Pic}^0_X,
$$
and such a complex $Q$ is given by $\Gamma _{\sigma *}\mls L$, where $\sigma $ is an automorphism of $X$ and $\mls L$ is a line bundle on $X$ such that the pair $(\sigma, \mls L)$ defines a point of $\text{Aut}^0_X\times \text{Pic}_X^0.$

The given object $F$ defines a morphism of stacks
\begin{equation}\label{E:2.3.1}
\mls R_X^0\rightarrow \mls D, \ \ (\sigma, \mls L)\mapsto (\sigma ^*F)\otimes \mls L,
\end{equation}
where $\mls D$ is as in the proof of \ref{L:5.4}.
The tangent space at the origin of $\mls R_X^0$ corresponds to isomorphism classes of deformations of $\Delta _*\mls O_X$ in $D(X\times X)$, which is given by 
$$
HH^1(X)= \text{Ext}^1(\Delta _{X*}\mls O_X, \Delta _{X*}\mls O_X),
$$
and the tangent space to $[F]\in \mls D(k)$ is given by 
$$
\text{Ext}^1(F, F).
$$
The map $\kappa _{F}$ is induced by \eqref{E:2.3.1} by passing to the tangent spaces.

Note also that since $\mls R_X^0$ is a $\mathbf{G}_m$-gerbe over $\mathbf{R}_X^0$ we can also view $HH^1(X)$ as the tangent space of $\mathbf{R}_X^0$ at the origin.
\end{pg}

\section{Proof of \ref{T:mainthm}}\label{S:section6}

It suffices to prove the theorem after making a base change to an algebraic closure of $k$, so without loss of generality we may assume that $k$ is algebraically closed.

We first show that if $F$ satisfies conditions (i) and (ii) then $F$ has the form indicated in \ref{T:mainthm}.

\begin{pg}\label{P:6.1} 
Let $\mathbf{S}_X\subset \mathbf{R}_X^0=A\times A^t$ be the stabilizer of the point $[F]\in P$ (recall that $P$ is the coarse space of the $\mathbf{G}_m$-gerbe $\mls P$).  The tangent space of $\mathbf{S}_X$ at the origin is then the kernel of $\kappa _{F}$. Since $\text{Ext}^1(F, F)$ has dimension $\leq d$ by assumption,  it follows that the dimension of $\mathbf{S}_X$ is equal to some integer $g\geq d$ (here we use the fact that $k$ has characteristic $0$).  Let $\mathbf{S}_X^0\subset \mathbf{S}_X$ be the connected component of the identity, and let 
$$
\mathbf{S}_X'\subset A
$$
be the image of $\mathbf{S}_X^0$ under the first projection and set
$$
\mathbf{K}:= \text{Ker}(\mathbf{S}_X^0\rightarrow \mathbf{S}_X').
$$
 So $\mathbf{S}_X'$ is an abelian subvariety of $A$ consisting of those points $a\in A$ for which $t_a^*F\simeq F\otimes \mls L$ for a line bundle $\mls L$ algebraically equivalent to $0$ (the reason for considering the space $P$ is to make this loose statement precise).  Let $g'$ denote the dimension of $\mathbf{S}_X'$.
\end{pg}

\begin{pg}
Let $Z\subset X$ be the set-theoretic support of $F$, viewed as a subscheme with the reduced structure.  Since $Z$ is invariant under $\mathbf{S}_X'$ every irreducible component of $Z$ has dimension $\geq g'$.

Let $Z'\subset Z$ be an irreducible component of dimension $d'$ and let $\widetilde Z\rightarrow Z'$ be a surjective morphism with $\widetilde Z$ smooth and projective over $k$ of dimension $d'$ (such a $\widetilde Z$ exists using the theory of alterations \cite{dejong}).  If $\widetilde {\mathbf{T}}$ denotes the Albanese torsor of $\widetilde Z$ then since $X$ is its own Albanese torsor the map $\widetilde Z\rightarrow X$ extends to a map
$$
\rho :\widetilde {\mathbf{T}}\rightarrow X.
$$
Since $Z'$ has dimension $d'$ it follows that this map has image of dimension $\geq d'$.  Taking duals we find that the image of 
$$
A^t = \text{Pic}^0_X\rightarrow \text{Pic}^0_{\widetilde {\mathbf{T}}}
$$
has dimension $\geq d'$.  

Now observe that the connected component of the identity $\mathbf{K}^0$ of $\mathbf{K}\subset A^t$ is in the kernel of this map.  Indeed let $\widetilde {F}$ be the pullback of $F$ to $\widetilde Z$.  Then for some integer $i$ the sheaf $\mls H^i(\widetilde {F})$ has generic rank $r_i\neq 0$.  Now if $\mls L$ is a line bundle on $X$ corresponding to a point of $\mathbf{K}^0$ then we have 
$$
\widetilde {F}\otimes \mls L|_{\widetilde Z}\simeq \widetilde {F}.
$$
Taking determinants we find that
$$
\text{det}(\mls H^i(\widetilde {F}))\otimes \mls L^{\otimes r_i}|_{\widetilde Z}\simeq \text{det}(\mls H^i(\widetilde {F})),
$$
and therefore $\mls L|_{\widetilde Z}$ is a torsion line bundle.  Since $\mathbf{K}^0$ is connected it follows that $\mls L|_{\widetilde Z}$ is trivial.

We conclude that the dimension of $\mathbf{K}$ is $\leq d-d'$.  Therefore
$$
g = g'+\text{dim}(\mathbf{K})\leq d'+(d-d') =d
$$
with equality if and only $g' = d'$ and $\text{dim}(\mathbf{K}) = d-d'$.  Since $g\geq d$ we conclude that $d' = g'$ and that each irreducible component of $Z$ is  a $\mathbf{S}'_X$-orbit.  Since $F$ is a point sheaf the support $Z$ is connected and we conclude that $Z$ is, in fact, a torsor under $\mathbf{S}_X'$.  
\end{pg}

\begin{pg} 
We now apply the results of the previous sections with \eqref{E:setup} the diagram
$$
\xymatrix{
\mathbf{S}_X^0\ar[d]_-u\ar[r]^-v& \mathbf{S}^{\prime t}\\
\mathbf{S}'.&}
$$
We then conclude from \ref{P:4.7} that $F$ has the desired form. 
\end{pg}

\begin{pg}
For the converse statement in \ref{T:mainthm} proceed as follows.
Let $h$ be the dimension of $Z$ so that $A/H$ has dimension $g-h$.  We have
$$
\text{\rm Ext}^*_X(i_*\mls E, i_*\mls E)\simeq \text{\rm Ext}^*_Z(i^*i_*\mls E, \mls E)\simeq H^*(Z, (i^*i_*\mls E)^\vee \otimes \mls E).
$$
This implies that 
$$
\text{\rm Ext}^0_X(i_*\mls E, i_*\mls E)\simeq H^0(Z, \mls E^\vee \otimes \mls E)\simeq k,
$$
since $\mls E$ is geometrically simple.  To calculate $\text{\rm Ext}^1$, note that by the proof of \ref{L:4.5} we have a distinguished triangle
$$
\mls E^\vee \otimes \mls E\rightarrow \tau _{\leq 1}(i^*i_*\mls E)^\vee \otimes \mls E\rightarrow (\mls E^\vee \otimes \mls E)^{\oplus (g-h)}[-1]\rightarrow \mls E^\vee \otimes \mls E[1].
$$
It follows that we an exact sequence
$$
\text{\rm Ext}^1_Z(\mls E, \mls E)\rightarrow \text{\rm Ext}^1_X(i_*\mls E, i_*\mls E)\rightarrow \text{\rm Ext}^0_Z(\mls E, \mls E)^{\oplus (g-h)}.
$$
Now by \cite[5.8]{MR498572} the dimension of $\text{Ext}^1_Z(\mls E, \mls E)$ is $h$ and since $\mls E$ is simple the dimension of $\text{Ext}^0(\mls E, \mls E)$ is $1$.  We conclude that the dimension of $\text{Ext}^1_X(i_*\mls E, i_*\mls E)$ is less than or equal to $g$ as desired.

This completes the proof of \ref{T:mainthm}. \qed
\end{pg}

\begin{example}\label{E:example6.5}
Let $k$ be an algebraically closed field, let $i:A\hookrightarrow B$ be a closed immersion of abelian varieties over $k$, and let $\mls E$ be a simple semi-homogeneous vector bundle on $A$.  For $F = i_*\mls E$ we can then describe the group scheme $\mathbf{S}_X$ in the above discussion as follows (a consequence of \ref{T:mainthm} is that this example is, in fact, general).  By \cite[5.6 and 5.8]{MR498572} (or use  \ref{T:2.7.1}) there is an isogeny $\pi :A'\rightarrow A$ and a line bundle $\mls N'$ on $A'$ such that $\alpha ':\lambda _{\mls N'}:A'\rightarrow A^{\prime t}$ restricts to an injection on $\text{Ker}(\pi )$ and such that $\mls E\simeq \pi _*\mls N'$.  The abelian variety $\mathbf{S}_X$ can then be described as follows.

The $k$-points of $\mathbf{S}_X$ are pairs $(b, [\mls L])\in B(k)\times B^t(k)$ such that there exists an isomorphism $t_b^*(F)\simeq F\otimes \mls L.$  Such an isomorphism necessarily preserves the scheme-theoretic support of $F$, which implies that $b\in A(k)$ and that we have an isomorphism $\sigma :t_b^*\mls E\simeq \mls E\otimes i^*\mls L$ on $A$.  Now we have 
$$
\mls E\otimes i^*\mls L\simeq (\pi _*\mls N')\otimes i^*\mls L\simeq \pi _*(\mls N'\otimes \mls L|_{A'}).
$$
By adjunction the map $\sigma $ corresponds to a map over $A'$
$$
\pi ^*t_b^*\mls E \simeq \oplus _{b'\in \pi ^{-1}(b)}t_{b'}^*\mls N'\rightarrow \mls N'\otimes \mls L|_{A'}.
$$
From this we conclude that there exists a unique element $b'\in \pi ^{-1}(b)$ such that $[\mls L]_{A'} = \alpha '(b')$ in $A^{\prime t}(k)$.  From this we conclude that there is a cartesian diagram
$$
\xymatrix{
\mathbf{S}_X\ar[r]\ar[d]& A^{\prime }\ar[d]^-{(\text{id}_{A'}, \alpha ')}\\
A'\times B^t\ar[r]& A^{\prime }\times A^{\prime t}.}
$$
with the map to $B\times B^t$ given by the composition
$$
\xymatrix{
\mathbf{S}_X\ar[r]& A'\times B^t\ar[r]^-{i\circ \pi \times \text{id}_{B^t}}& B\times B^t.}
$$
Equivalently, we have a cartesian diagram
$$
\xymatrix{
\mathbf{S}_X\ar[r]\ar[d]& A'\ar[d]^-{\alpha '}\\
B^t\ar[r]& A^{\prime t}.}
$$

From this description it follows that $\mathbf{S}_X$ is connected.  Indeed dualizing the exact sequence
$$
0\rightarrow A\rightarrow B\rightarrow B/A\rightarrow 0
$$
we get an exact sequence (using that $B\rightarrow B/A$ has geometrically connected fibers)
$$
0\rightarrow (B/A)^t\rightarrow B^t\rightarrow A^t\rightarrow 0.
$$
We therefore get an exact sequence
$$
0\rightarrow (B/A)^t\rightarrow \mathbf{S}_X\rightarrow A^t\times _{A^{\prime t}, \alpha '}A'\rightarrow 0.
$$
Now $A^{t}\rightarrow A^{\prime t}$ is the cover corresponding by duality to $\text{Ker}(\pi )\hookrightarrow A'$.  More concretely, if we write $\mls P_{A', a'}$ for the restriction of the Poincar\'e bundle on $A'\times A^{\prime t}$ to $\{a'\}\times A^{\prime t}$ then $A^t$ is described, as a scheme over $A^{\prime t}$, by
$$
\oplus _{a'\in \text{Ker}(\pi )}\mls P_{A', a'},
$$
with algebra structure given by the biextension structure on $\mls P_{A'}$.  It follows from this that the pullback of this cover along $\alpha '$ is the cover of $A'$ given by
$$
\oplus _{a'\in \text{Ker}(\pi )}t_{a'}^*\mls N'\otimes \mls N^{\prime -1}.
$$
That is, $A^t\times _{A^{\prime t}, \alpha '}A'\rightarrow A'$ is the cover given by the composition
$$
\xymatrix{
\text{Ker}(\pi )\ar[r]& A'\ar[r]^-{\alpha '}& A^{\prime t}.}
$$
Since this map is injective by assumption it follows that $A^t\times _{A^{\prime t}, \alpha '}A'$, and therefore also $\mathbf{S}_X$, is connected.  From this it also follows that we have an exact sequence
$$
0\rightarrow A^{\prime t}/\alpha '(\text{Ker}(\pi ))\rightarrow \mathbf{S}_X^t\rightarrow (B/A)\rightarrow 0.
$$
\end{example}

\section{Lagrangians}\label{S:section7}

As noted in the introduction our main result \ref{T:mainthm} is closely related to Polishchuk's work in \cite{Polishchuk2, Polishchuk3}.  In this section we discuss this connection in more detail.

Let $k$ be an algebraically closed field of characteristic $0$.

\begin{pg} Recall that if $\mathbf{S}$ is an abelian variety over $k$ and $\mls N$ is a line bundle on $\mathbf{S}$ then there is an associated biextension $\Lambda (\mls N)$ \cite[VII]{SGA7I} (see also \cite[II.10.3]{Polishchuk}).  This biextension is obtained from the Poincar\'e bundle $\mls P$ on $\mathbf{S}\times \mathbf{S}^t$, endowed with its canonical structure of a biextension, by pullback along the map $\text{id}_{\mathbf{S}}\times \lambda _{\mls N}:\mathbf{S}\times \mathbf{S}\rightarrow \mathbf{S}\times \mathbf{S}^t.$

If $A/k$ is an abelian variety then a \emph{Lagrangian pair} in $A\times A^t$ \cite[2.2.2]{Polishchuk2} consists of the following data:
\begin{enumerate}
    \item [(i)] An abelian subvariety $i:\mathbf{S}\subset A\times A^t$ of the same dimension as $A$.
    \item [(ii)] A line bundle $\mls N$ on $\mathbf{S}$ together with an isomorphism $\Lambda (\mls N)\simeq i^*\mls B_A$ of biextensions, where $\mls B_A$ is the biextension over $(A\times A^t)\times (A\times A^t)$ obtained by pulling back the Poincar\'e bundle along the projection to the first and fourth factor.
\end{enumerate}
By \cite[2.4.5]{Polishchuk2} (applied with $(Y, \alpha ) = (\{0\}\times A^t, \mls O_{\{0\}\times A^t})$) and \cite[2.3.2]{Polishchuk2} a Lagrangian pair $(\mathbf{S}, \mls N)$ gives rise to a coherent sheaf $S_{(\mathbf{S}, \mls N)}$ on $A$ equipped with an isomorphism over $A\times \mathbf{S}$
\begin{equation}\label{E:7.1.1}
m^*S_{(\mathbf{S}, \mls N)}\otimes \text{pr}_{13}^*\mls P_A^{-1}\otimes \text{pr}_2^*\mls N^{-1}\simeq \text{pr}_1^*S_{(\mathbf{S}, \mls N)},
\end{equation}
where $m:A\times \mathbf{S}\rightarrow A$ is the composition of the projection map $A\times \mathbf{S}\rightarrow A\times A$ and the addition map, $\text{pr}_{13}^*\mls P_A$ is the pullback of the Poincar\'e bundle along the first and third projections $A\times \mathbf{S}\rightarrow A\times A^t$, and $\text{pr}_2$ is given by the projecton $A\times \mathbf{S}\rightarrow \mathbf{S}$ followed by $u:\mathbf{S}\rightarrow A$. 
\end{pg}

\begin{rem} The formula \eqref{E:7.1.1} differs from the one in \cite[2.4]{Polishchuk2} since we are taking inverses of the line bundles, so what we denote by $S_{(\mathbf{S}, \mls N)}$ corresponds to the dual of the objects in \cite{Polishchuk2}.  This sign convention will be useful in the proofs below.
\end{rem}

\begin{lem} There exists a subabelian variety $H\subset A$ and an $H$-torsor $i:Z\hookrightarrow A$ such that $S_{(\mathbf{S}, \mls N)}\simeq i_*\mls E$ for a simple semi-homogeneous vector bundle $\mls E$ on $Z$.  In particular, $S_{(\mathbf{S}, \mls N)}$ satisfies conditions (i) and (ii) in \ref{T:mainthm}.
\end{lem}
\begin{proof}
Let $\mls D$ be as in the proof of \ref{L:5.4}.  Again by \cite[\href{https://stacks.math.columbia.edu/tag/0BDL}{Tag 0BDL}]{stacks-project} there is a locally closed substack $\mls D'\subset \mls D$ classifying objects $K\in D(A)$ with $\text{End}(K) = k$,  and $\mls D'$ is a $\mathbf{G}_m$-gerbe over an algebraic space $D'$.  The object $S_{(\mathbf{S}, \mls N)}$ defines an object in $\mls D'(k)$ since it is ``endosimple'' in the terminology of \cite{Polishchuk2}.  We can therefore consider the stabilizer $\mathbf{S}_X\subset A\times A^t$ of the induced point of $D'$, as in \ref{P:6.1}.  By \cite[2.4.10]{Polishchuk2} this stabilizer equals $\mathbf{S}$ and therefore has dimension equal to the dimension of $A$.  Applying the argument of section \ref{S:section6}, noting that condition (ii) is used in the argument to conclude that the dimension of $\mathbf{S}_X$ is $\geq \text{dim}(A)$, we then obtain the lemma.
\end{proof}

\begin{thm}\label{T:7.3} The map of sets
$$
\xymatrix{
\{\text{Lagrangian pairs $(\mathbf{S}, \mls N)$}\}\ar[d]^-{(\mathbf{S}, \mls N)\mapsto S_{(\mathbf{S}, \mls N)}}\\
\{\text{coherent sheaves $\mls F$ satisfying (i) and (ii) in \ref{T:mainthm}}\}}
$$
is a bijection.
\end{thm}
\begin{proof}
Starting with a coherent sheaf $\mls F$ satisfying (i) and (ii) in \ref{T:mainthm} we construct the Langrangian pair $(\mathbf{S}, \mls N)$ mapping to it as follows.

Let $Z\subset A$ be the support of $\mls F$, so $Z$ is a torsor under an abelian variety $i:H\subset A$, which is a quotient $\mathbf{S}_X\rightarrow H$ of the subgroup scheme $\mathbf{S}_X\subset A\times A^t$ defined as in \ref{P:6.1}.
Let $m :A\times \mathbf{S}_X\rightarrow A$ be the action map given by 
$$
(a, a', [\mls L])\mapsto a+a',
$$
and let $\mls B$ be the line bundle on $A\times \mathbf{S}_X$ given by pulling back the Poincar\'e bundle on $A\times A^t$ along $\text{pr}_{13}:A\times \mathbf{S}_X\rightarrow A\times A^t$.  

To prove the theorem it suffices to prove the following claim, the argument for which occupies the remainder of the proof.

\begin{claim} There exists a unique line bundle $\mls N$ on $\mathbf{S}_X$ for which $(\mathbf{S}_X, \mls N)$ is a Lagrangian pair and such that there exists an isomorphism of sheaves on $A\times \mathbf{S}_X$ as in \ref{E:7.1.1} with $S_{(\mathbf{S}_X, \mls N)}$ replaced by $\mls F$.
\end{claim}

If $\mls N_1$ and $\mls N_2$ are two invertible sheaves on $\mathbf{S}_X$ for which $(\mathbf{S}_X, \mls N_i)$ ($i=1,2$) are both Lagrangian pairs then the biextension $\Lambda (\mls N_1\otimes \mls N_2^{-1})$ is trivial and therefore $\mls N_2\simeq \mls N_1\otimes \mls R$ for a translation invariant line bundle $\mls R$ on $\mathbf{S}_X$ (recall that $\Lambda (\mls N)$ is the biextension over $\mathbf{S}_X\times \mathbf{S}_X$ given by $m^*\mls N\otimes \text{pr}_1^*\mls N^{-1}\otimes \text{pr}_2^*\mls N^{-1}$).  From the formula \eqref{E:7.1.1} for both $\mls N_1$ and $\mls N_2$ we get an isomorphism
$\text{pr}_1^*\mls F\simeq \text{pr}_1^*\mls F\otimes \text{pr}_2^*\mls R,$
which defines a nonzero map
$$
\text{pr}_2^*\mls R\rightarrow \mls RHom (\text{pr}_1^*\mls F, \text{pr}_1^*\mls F).
$$
By adjunction and using property (i) in \ref{T:mainthm} we get a nonzero map
$$
\mls R\rightarrow \text{pr}_{2*}\mls R^0\mls Hom(\text{pr}_1^*\mls F, \text{pr}_1^*\mls F)\simeq \mls O_{\mathbf{S}_X}.
$$
Since $\mls R$ is algebraically equivalent to $0$ we conclude that this map is an isomorphism and that $\mls R$ is trivial.  This proves the uniqueness part of the claim.

For the existence part, note that there exists an isogeny $\pi :H'\rightarrow H$, a line bundle $\mls N'$ on $H'$ for which  the map $\lambda _{\mls N'}:H'\rightarrow H^{\prime t}$ is injective on $\text{Ker}(\pi )$, and a point $a\in A(k)$ such that  $\mls F = t_a^*i_*\pi _*\mls N'.$ 

By the discussion in \ref{E:example6.5} we have $\mathbf{S}_X = A^t\times _{H^{\prime t}, \alpha '}H'.$  Now consider the commutative diagram
$$
\xymatrix{
H'\times \mathbf{S}_X\ar[d]_-{\text{id}_{H'}\times \text{pr}_3}\ar[r]^-{\text{pr}_{12}}& H'\times H'\ar[d]^-{\text{id}\times \alpha '}\\
H'\times A^t\ar[r]^-{\text{id}_{H'}\times i^t}\ar[d]& H'\times H^{\prime t}\\
A\times A^t,}
$$
where the square is cartesian.  Now we have an isomorphism of biextensions
$\mls P_{A}|_{H'\times A^t}\simeq (1\times i^t)^*\mls P_H,$
so we get an isomorphism of biextensions $\text{pr}_{13}^*\mls P_{A}\simeq \text{pr}_{12}^*\Lambda (\mls N')$ over $H'\times \mathbf{S}_X.$ 
$\text{pr}_{13}^*\mls P_{A}\simeq \text{pr}_{12}^*\Lambda (\mls N')$
Therefore if $\mls N$ denotes the pullback of $\mls N'$ along the projection $\mathbf{S}_X\rightarrow H'$ then $(\mathbf{S}_X, \mls N)$ is a Lagrangian pair.  Furthermore, writing out this isomorphism we get
$$
\text{pr}_{13}^*\mls P_A\simeq m^*\mls N'\otimes \text{pr}_1^*\mls N^{\prime -1}\otimes \text{pr}_{2}^*\mls N^{\prime -1},
$$
which rearrages to
$$
m^*\mls N^{\prime }\otimes \text{pr}_{13}^*\mls P_A^{-1}\otimes \text{pr}_2^*\mls N^{-1}\simeq \text{pr}_1^*\mls N'.
$$
Pushing forward to $A\times \mathbf{S}_X$ we see that the claim holds for $i_*\pi _*\mls N'$.

To handle the case of $\mls F = t_a^*i_*\pi _*\mls N'$ for some $a\in A$ consider the commutative diagram
$$
\xymatrix{
A\times \mathbf{S}\ar[r]^-m\ar[d]_-{t_a\times \text{id}}& A\ar[d]^-{t_a}\\
A\times \mathbf{S}\ar[r]^-m& A,}
$$
and note that using the biextension structure on $\mls P_A$ we have
$$
(t_a\times \text{id})^*\text{pr}_{13}^*\mls P_A\simeq \text{pr}_{13}^*\mls P_A\otimes \text{pr}_3^*\mls P_{A, a},
$$
where $\mls P_{A, a}$ is the line bundle on $A^t$ obtained by restriction $\mls P_A$ to $\{a\}\times A^t$.  From this it follows that taking $\mls N$ equal to the pullback of $\mls N'$ tensored with the pullback of $\mls P_{A, a}^{-1}$ verifies the claim for $t_a^*i_*\pi _*\mls N'$.  

This completes the proof of the claim and \ref{T:7.3}.
\end{proof}

\providecommand{\bysame}{\leavevmode\hbox to3em{\hrulefill}\thinspace}
\providecommand{\MR}{\relax\ifhmode\unskip\space\fi MR }
\providecommand{\MRhref}[2]{%
  \href{http://www.ams.org/mathscinet-getitem?mr=#1}{#2}
}
\providecommand{\href}[2]{#2}

\end{document}